\documentclass{amsproc}

\usepackage{amssymb}

\usepackage{graphicx}

\usepackage{amsthm}
\usepackage{mathrsfs} 
\usepackage{duerer}
\usepackage{txfonts}
\usepackage{epsfig}
\usepackage{caption}

\newtheorem{theorem}{Theorem}[section]

\newtheorem{proposition}{Proposition}

\theoremstyle{definition}
\newtheorem{definition}[theorem]{Definition}
\newtheorem{example}[theorem]{Example}

\theoremstyle{remark}

\newtheorem{notation}{Notation}

\numberwithin{equation}{section}

\newcommand{\gen}{\multimapinv}
\newcommand{\eq}{\multimapboth}
\newcommand{\gv}{\coloneq}
\newcommand{\eqv}{\coloneq:}
\newcommand{\eqar}{\stackrel{\mbox{\rm ar}}\longleftrightarrow}

\begin{document}

\title{Geometric and arithmetic relations concerning origami}


\author{Jordi Gu\`ardia}
\address{ Departament de Matem\`atica Aplicada IV, Escola Polit\`ecnica Superior d'Enginyeria de Vilanova i la Geltr\'u, Universitat Polit\`ecnica de Catalunya.  Avda. V\'ictor Balaguer S/S - 08800
 Vilanova i la Geltr\'u, Spain}
\curraddr{}
\email{guardia@ma4.upc.edu}
\thanks{The first author is partially supported by MTM2012-34611 from Spanish MEC}

\author{Eul\`alia Tramuns}
\address{Departament de Matem\`atica Aplicada IV, Universitat Polit\`ecnica de Catalunya, Barcelona Tech.
C/ Jordi Girona 1-3, M\`odul C3, Campus Nord, 08034 Barcelona, Spain}
\curraddr{}
\email{etramuns@ma4.upc.edu}
\thanks{The second author is partially supported by MTM2011-28800-C02-01 from Spanish MEC}

\subjclass[2000]{Primary }

\date{}

\begin{abstract}
We present a formalization of geometric instruments that considers separately geometric and arithmetic aspects of them. We introduce the concept of \emph{tool}, which formalizes a physical instrument as a set of \emph{axioms} representing its geometric capabilities. We also define a \emph{map} as a tool together with a set of points and curves as an initial reference. We rewrite known results using this new approach and give new relations between origami and other instruments, some obtained considering them as tools and others considering them as maps.
\end{abstract}

\maketitle

\section{Introduction}

The determination of constructible numbers with origami is a problem with an interesting development, that was completely solved only after the axiomatization proposed by Huzita-Justin (cf. \cite{Huzita}, \cite{Justin}). The work by Alpering-Lang (\cite{AlperinLang}) proved that the list of possible one-fold axioms was complete and settled a new scenario, where the role of new axioms was still emphasized. The axiomatic viewpoint seems a natural perspective for the study of other geometrical instruments. 

In this work we present a general purpose formal language for the axiomatization of geometrical instruments. Our formalization takes into account both the geometric and arithmetic properties of the instruments: the concept of {\em tool} formalizes an instrument as a set of axioms, while the concept of {\em map} formalizes the constructible points and curves of the instrument. Our formalization provides a natural frame to express well-known results, but also leads to new relations between instruments.

The key concepts of our language are introduced in section 2. In section $3$, we define \emph{geometric equivalence} and \emph{virtual equivalence of tools} and prove relations between the tools described in this work. In section $4$ we present an equivalence relation between maps and define an \emph{arithmetical equivalence of tools}. We conclude with an arithmetic classification of the tools described.

A more extended version of this work can be found in \cite{TramunsThesis} as an evolution of a previous work (cf. \cite{Tramuns}).

\newpage
\section{Definitions}
\subsection{Axioms}
\begin{definition}
A \emph{construction axiom \bf{C}}
is an elementary geometric process that generates a finite ordered set of curves from a non empty, finite and ordered set
of points and curves.

An \emph{intersection axiom {\bf I}}
is a an elementary geometric process that generates a finite ordered set of points from a non empty, finite and ordered set of curves and points.

\end{definition}
The notation we will use for axioms is $O_1,O_2,\dots, O_r=\textbf{AxiomName}(I_1,I_2,\dots, I_s)$ where the $I_j$ are the given curves and points and the  $O_i$ are the elements generated by the axiom.

We collect in the Annex a list of basic axioms which formalize the common tasks performed by the geometric instruments that will be considered in this work (mainly ruler, compass, origami). A more comprehensive list of axioms can be found in \cite{TramunsThesis}.


\subsection{Tools}

We formalize geometric instruments in terms of their capabilities, i.e., by means of the axioms they can perform:

\begin{definition}
A \emph{tool} $T$ is a couple $\langle
\mathcal{C},\mathcal{I}\rangle$, where
 $\mathcal{C}$ is a non empty finite set of construction axioms and $\mathcal{I}$ is a non empty finite set of
 intersection axioms.
\end{definition}


Some basic examples of tools are:

\begin{itemize}
\item Ruler $${\rm {\bf \mathcal R}:=\langle {\bf \{Line\},\{LineIntersect\}\rangle}}$$
\item Compass $$
{\bf \mathcal C}:=\langle{\bf \{Circle,
RadiusCircle\},\{CircleIntersect\}\rangle}$$
\item Ruler and compass $$\begin{array}[b]{r@{}l} {\bf \mathcal{RC}:=\langle}&
{\rm\bf\{Line, Circle, RadiusCircle\}},\\
&  {\rm\bf\{ LineIntersect, CircleIntersect,
LineCircleIntersect\}\rangle}.
\end{array}$$
\item Euclidean compass
\begin{center}
${\bf\mathcal{EC}:=}\langle$\{{\bf Circle}\}, \{{\bf
CircleIntersect}\}$\rangle$.
\end{center}
\item Ruler and Euclidean compass
$$\begin{array}[b]{r@{}l} {\bf \mathcal{REC}}:=\langle&
{\rm\bf\{Line, Circle\}},\\
&  {\rm\bf\{ LineIntersect, CircleIntersect,
LineCircleIntersect\}\rangle}.
\end{array}$$
\item Origami \begin{center}
${\bf\mathcal{O}:=}\langle$\{{\bf Line, PerpendicularBisector,
Bisector, Perpendicular, Tangent, CommonTangent,
PerpendicularTangent}\}, \{{\bf LineIntersect}\}$\rangle$.
\end{center}
\item Thalian Origami (\cite{Alperin})
\begin{center}
${\bf\mathcal{TO}:=}\langle$\{{\bf Line, PerpendicularBisector}\}, \{{\bf LineIntersect}\}$\rangle$.
\end{center}
\item Pythagorean origami (\cite{Alperin})
\begin{center}
${\bf\mathcal{PO}:=}\langle$\{{\bf Line, PerpendicularBisector, Bisector}\},
\{{\bf LineIntersect}\}$\rangle$.
\end{center}
\item Conics (\cite{Videla})
$$\begin{array}[b]{r@{}l} {\bf \mathcal{CO}}:=\langle&
{\rm\bf\{Line, Circle, RadiusCircle, Conic\}},{\rm\bf\{ LineIntersect, }\\
& {\rm\bf CircleIntersect,LineCircleIntersect,ConicLineIntersect,}\\
& {\rm\bf ConicCircleIntersect, ConicIntersect\}\rangle}.
\end{array}$$

\end{itemize}

\subsection{Constructions}

\begin{definition}
A \emph{construction} of a set of points and lines $V$  from
$U_0$ is a finite sequence
$$
C(U_0;V)=\{O_1=A_1 (U_1), ..., O_n=A_n (U_n)\},
$$
where
\begin{itemize}
\item[-] $U_0$ is an initial ordered non empty set of points and curves;
\item[-] $A_1,\dots, A_n$ are axioms;
\item[-] Every $U_k$ is a subset of $U_0\cup \dots\cup U_{k-1}\cup O_1\cup \dots\cup O_{k-1}$;
\item[-] $V\subset O_1\cup \dots \cup O_n$, but $V\not\subset O_1\cup\dots \cup O_{n-1}$.
\end{itemize}

We  say that $C(U_0;V)$ is a construction with the tool $\mathcal E=\langle{\mathcal C},{\mathcal I}\rangle$
if $A_1,\dots, A_k\in {\mathcal C}\cup{\mathcal I}$, and we write
$C(U_0;V)\in {\mathcal E}$ in this case.

\end{definition}

\begin{notation}
To simplify the notation, we enumerate the elements of the sets $U_0$, $V$ in a single list, using a semicolon to separate the last element of $U_0$ and the first of $V$.
\end{notation}

\begin{example}
Given a line $\ell$ and a point $P$ not on $\ell$, the construction $\textbf{Parallel}$ generates the parallel $\ell_2$ to the line $\ell$ passing through the point $P$ (:
$$
\begin{array}{r@{}l}
\textbf{\emph{Parallel}}(P,\ell;\ell_2)=\left\{\right.&
\ell_1=\textbf{Perpendicular}(\ell,P),\\
& \left.\ell_2=\textbf{Perpendicular}(\ell_1,P)\right\}.
\end{array}
$$
Clearly $\textbf{\emph{Parallel}}(P,\ell;\ell_2)\in {\mathcal O} $.
\end{example}

A wide catalog of constructions is given in \cite{TramunsThesis}.

\subsection{Maps}

\begin{definition}
A \emph{map} is a pair ${\textdubf M}=({\mathcal E},\mathcal{U}_0)$
composed by a tool ${\mathcal E}$  and a non empty finite  {\em initial set} $\mathcal{U}_0$ of points and curves.
\end{definition}

\begin{notation}
Given a set $U$ of points and curves, we will write  $U=[C,P]$ to specify its subset $C$ of curves and its subset $P$ of points.
\end{notation}

\begin{definition}
Let ${\textdubf M}=({\mathcal E},\mathcal{U}_0)$ be a map
with ${\mathcal E}=\langle \mathcal{C},\mathcal{I}\rangle$. The
\emph{sequence of layers}
$\mathcal{U}_n=\{[\mathfrak{C}_n,\mathfrak{P}_n]\}_{n\in\mathbb{N}}$ is the sequence defined by:
\begin{itemize}
\item[i)] $\mathcal{U}_0=[\mathfrak{C}_0,\mathfrak{P}_0]$,
\item[ii)] $\mathfrak{C}_n$ is the union of $\mathfrak{C}_{n-1}$ with the set of curves obtained
applying all construction axioms from $\mathcal{C}$
in all possible ways to the elements of $\mathcal{U}_{n-1}$,

\item[iii)] $\mathfrak{P}_n$ is the union of $\mathfrak{P}_{n-1}$ with the set of points that is obtained applying all intersection axioms from
 $\mathcal{I}$
in all possible ways to the elements of
$[\mathfrak{C}_{n},\mathfrak{P}_{n-1}]$.

\end{itemize}
We write
$\mathcal{U}^{\textdubf M}=[\mathfrak{C}^{\textdubf M},\mathfrak{P}^{\textdubf M}]:=\cup_{n=0}^\infty
\mathcal{U}_n$ to denote the \emph{set of constructible points and curves with
${\textdubf M}$}.
A map is \emph{infinite} if ${\mathcal U}^{{\textdubf M}}$ is infinite.
\end{definition}

Table \ref{cataleg_conjuntsfinals} describes the set of constructible points  
of the tools introduced in previous section. As usual, $\mathbb{P}$ denotes the Pythagorean closure of $\mathbb{Q}$ (i.e., the smallest extension of $\mathbb{Q}$ where every sum of two squares is a square); the field of Euclidean numbers (i.e., the smallest subfield of $\overline{\mathbb{Q}}$ closed under square roots) is denoted by  $\mathscr{C}$, and $\mathscr{O}$ is the field of origami numbers (i.e., the smallest subfield of $\overline{\mathbb{Q}}$ which is closed under the operations of taking square roots, cubic roots and complex conjugation).

\begin{table}[h!]\centering
\begin{tabular}{|l|c|c|c|}
\hline
Map & Initial set & $\mathfrak{P}^{\textdubf M}$ & References\\
\hline Ruler  & $\{1,2,i,2i\}$& $\mathbb{Q}(i)$& \cite[page 79]{Martin} \\
Compass   & $\{0,1\}$ & $\mathscr{C}$&  \cite[chap. 3]{Martin} \\
Ruler and compass   & $\{0,1\}$  &$\mathscr{C}$&   \cite[page 261]{Cox}\\
Euclidean compass   & $\{0,1\}$ & $\mathscr{C}$ & \cite[page 7]{Martin}\\
Ruler and euclidean compass   & $\{0,1\}$ &$\mathscr{C}$ & \cite[page 7]{Martin}  \\
Origami  & $\{0,1\}$& $\mathscr{O}$& \cite{Alperin}\\
Pythagorean Origami & $\{0,1\}$ &$\mathbb{P}(i)$&\cite[Theorem 3.3]{Alperin}  \\
Conics   & $\{0,1\}$&$\mathscr{O}$ & \cite{Videla} \\
\hline
\end{tabular}
\caption{Sets $\mathfrak{P}^{\textdubf M}$ for different maps}
\label{cataleg_conjuntsfinals}
\end{table}
For a description of the set of constructible points of the Thalian origami map, see \cite[Theorem 3.3]{Alperin}.

\section{Geometric relations between tools}

Since we have defined tools in terms of their capabilities, it is natural to classify them according to this philosophy. We will consider the arithmetic capabilities in the next section, and we concentrate now on the geometric capabilities. It seems reasonable to say that two tools are equivalent if they can solve the same problems, but this has to be carefully defined. 
For instance, the second problem of Euclid is normally formulated as it follows: given three points $A, B, C$ in general position, one has to determine a point $D$ such that the segments $AD$ and $BC$ are congruent. Clearly, the point $D$ is not uniquely determined, and indeed, there exist several different constructions generating different solution points. This kind of situation leads to introduce the following definition:

\begin{definition}
We say that the constructions $C(U_0;V)$ and $C'(U_0;V')$ are
\emph{equivalent}, $C(U_0;V)\sim C'(U'_0;V')$,
if $V$ and $V'$ have the same geometric links with $U_0$.

A \emph{problem} $P(U_0;V)$ is an equivalence class of constructions
under this relation.

A \emph{solution} of the problem $P(U_0;V)$ is any representant of this
equivalence class, that is, any construction
$C(U_0;V)$ of $V$ from $U_0$.

A problem can be solved with $\mathcal E$ if it has a solution
$C(U_0;V)\in \mathcal E$.
\end{definition}

\begin{definition}
The tool $\mathcal E$ generates the tool $\mathcal E'$,  $\mathcal E\gen \mathcal E'$ if any problem that can be solved with $\mathcal E'$ can also be solved with $\mathcal E$.
Two tools ${\mathcal E}$ and ${\mathcal E'}$ are {\em geometrically equivalent},  $\mathcal E\eq \mathcal E'$ if they solve the same problems.
\end{definition}

Obviously, if the set of axioms of a tool  $\mathcal E'$ is a subset of the axioms of the tool $\mathcal E$, then $\mathcal E\gen \mathcal E'$. Hence we have the following relations



$$\mathcal CO \gen \mathcal RC \gen \mathcal REC \gen {\mathcal R};$$
$${\mathcal O} \gen \mathcal PO \gen \mathcal TO \gen {\mathcal R}.$$

\begin{proposition}
The tool $\mathcal TO$ does not generate the tools $\mathcal PO$, ${\mathcal O}$.
\end{proposition}

\begin{proof}
It is enough to see that $\mathcal TO$ does not generate $\mathcal PO$.
Let us suppose we have two perpendicular lines and their point of intersection constructed. With the tool $\mathcal PO$ and the use of axiom {\bf Bisector} we can construct the bisectors of this couple of lines.
However with the tool $\mathcal TO$ we cannot construct neither a new line, nor a new point.
\end{proof}

While it is evident that the compass
$\mathcal C$ does not generate the ruler and compass $\mathcal RC$ since it cannot construct lines, both tools generate the same points from the initial set $\{0,1\}$. To describe this situation in a more general setting we introduce the following definitions:

\begin{definition}
A construction of points with points  (CPP) is a construction $CPP(U_0;V)$
where $U_0$ and $V$ contain only points.
\end{definition}

\begin{definition}
The tool $\mathcal E$
\emph{generates virtually} the tool $\mathcal E'$, $\mathcal E\gv\mathcal E'$, if any  CPP with
$\mathcal E'$ is equivalent to a construction with $\mathcal E$.
The tools ${\mathcal E}$ and ${\mathcal E'}$ are \emph{virtually
equivalents},
$\mathcal E\eqv\mathcal E'$ if $\mathcal E\gv\mathcal E'$ and $\mathcal E'\gv\mathcal E$.
\end{definition}

\begin{theorem}
The Origami tool  ${\mathcal O}$ generates virtually the ruler and Euclidean compass tool $\mathcal REC$.
\end{theorem}

\begin{proof} It is enough to describe constructions of the intersection of a line with a circle and of the intersection of two circles with ${\mathcal O}$.  We can find them in \cite{Geretschlager}. \end{proof}

As a more advanced example of construction, we describe the intersection of a circle and a line in our formal language. The construction \textbf{\emph{LineCircleIntersectOrigami}}(A,B,C,D;E,F)
generates
the intersection points of the line through the points $A$ and $B$ with the circle with center $C$ passing through $D$. 

This construction requires the point $C$ to be exterior to the parabola with focus $D$ and directrix the line through $A$ and $B$.
$$
\begin{array}{r@{}l}
\textbf{\emph{LineCircleIntersectOrigami}}(A,B,C,D;E,F)=\left\{\right.
& \ell_1= \textbf{Line}(A,B),\\
& \ell_2,\ell_3= \textbf{Tangent}(\ell_1,D,C),\\
& \ell_4=\textbf{Perpendicular}(\ell_2,D),\\
& \ell_5=\textbf{Perpendicular}(\ell_3,D),\\
& E= \textbf{LineIntersect}(\ell_1,\ell_4),\\
&  \left. F=\textbf{LineIntersect}(\ell_1,\ell_5)\right\}.
\end{array}
$$

Figure \ref{GeometricRelations} summarizes the main relations between tools.
\begin{figure}[h!]
\begin{center}
\epsfig{file=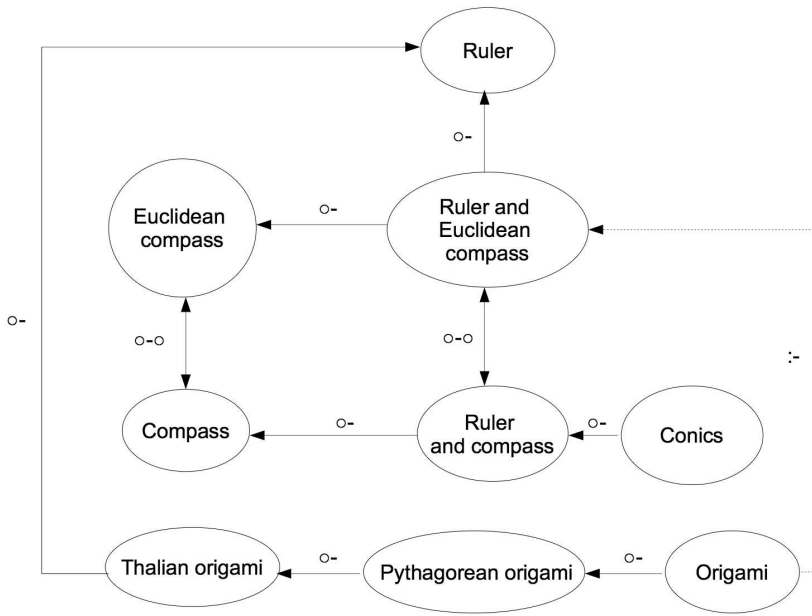,width=11.5cm}
\end{center}
\caption{Geometric relations between tools}
\label{GeometricRelations}
\end{figure}

\section{Arithmetic relations between maps and tools}

\begin{definition}
The maps $\textdubf M$ and $\textdubf M'$ are equivalent if $\mathfrak{P}^\textdubf M=\mathfrak{P}^{\textdubf M'}$.
\end{definition}

Figure \ref{Equivalent maps} shows the relations between sets of points of maps of Table \ref{cataleg_conjuntsfinals}.
\begin{figure}[h!]
\begin{center}
\epsfig{file=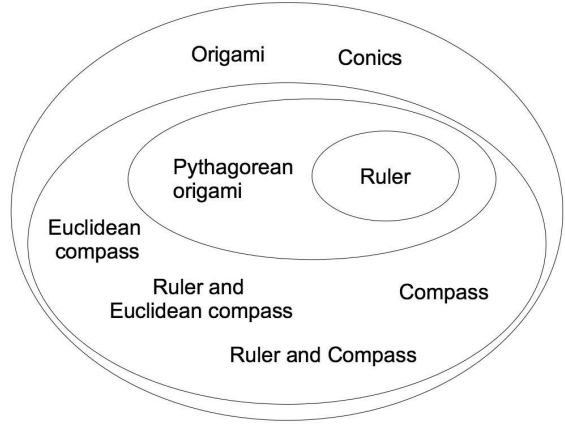,width=8cm}
\end{center}
\caption{Equivalence relations between maps}
\label{Equivalent maps}
\end{figure}

\begin{definition}
Two tools $\mathcal E, \mathcal E'$ are \textbf{\emph{arithmetically equivalents}}, $\mathcal E\eqar\mathcal E'$, if there exists finite sets of points ${\mathcal U}_0, {\mathcal U}_0'$ such that:
\begin{itemize}
\item[i)] The maps ${\textdubf M}=({\mathcal E},{\mathcal U}_0)$ and ${\textdubf M'}=({\mathcal E'},{\mathcal U}_0')$ are infinite and equivalent. 
\item[ii)] The construction of the set $\mathcal U^{\textdubf M}$ (resp. $\mathcal U^{\textdubf M'}$) needs the application of all the axioms of $\mathcal E$ (resp. $\mathcal E'$).
\end{itemize}
\end{definition}

In order to determine whether two given tools are arithmetically equivalent  two different approaches can be taken: one can consider the geometric properties of the tools  or one can associate particular maps to the tools and relate them. The following results illustrate these ideas.

\begin{theorem}[Mohr-Mascheroni]
\label{Mohr-Mascheroni}
The tools $\mathcal C$ and $\mathcal RC$ are arithmetically equivalent.
\end{theorem}

\begin{proof}
The first step consists in proving that we obtain the same points with $\mathcal RC$ than doing inversions of points respect to circles.
Then we have to prove that any point obtained from an inversion respect to a circle can be constructed with $\mathcal C$ in a finite, arbitrary high, number of steps. The details can be found in \cite{TramunsThesis}.
\end{proof}


\begin{theorem}[Poncelet-Steiner]
\label{poncelet-steiner}
The maps $\textdubf{RC}$ and
$$
\textdubf{RP}:=(\langle\{{\bf Line}\}, \{{\bf LineIntersect,
LineUnitCircleIntersect}\}\rangle,\mathcal{U}_0=\{0, 2, 2i,
X^2+Y^2=1\}).
$$
are equivalent.
\end{theorem}

\begin{proof} 
Of course, a totally geometric proof can be given (\cite[page 192]{Har}), but we present here an arithmetic proof in our language, following \cite[page 98]{Martin}.
It is clear that $1,i\in \mathfrak{P}^{\textdubf{RP}}$, and thus $\mathbb{Q}(i)=\mathfrak{P}^{\textdubf {R}}\subset\mathfrak{P}^{\textdubf {RP}}\subset\mathfrak{P}^{\textdubf {RC}}=\mathscr{C}$. Since $\mathscr{C}$ is the smallest extension of $\mathbb{Q}$ closed under square roots, it is sufficient to see that $\mathfrak{P}^{\textdubf{RP}}$ is closed under square roots. The construction of the square roots of a complex number reduces to the construction of the bisector of two lines through the origin and the construction of the square root of a positive real number.

Given two points $A, B\in\mathfrak{P}^{\textdubf{RP}}$ on the unit circle, let $C$ be the second point of intersection of the diameter through $A$ with the unit circle. Then the angle $AOB$ is twice the angle $OCB$.
The equality
$$
\sqrt{r}=\left(\frac{r+1}2\right)\sqrt{1-\left(\frac{r-1}{r+1}\right)^2}
$$
shows that we only need the construction of the square root of numbers of the form  $1-c^2$ with $c\in(-1,1)$, and this consists in constructing a point on the circle having $1-c^2$ as a x-coordinate. This can be done constructing the perpendicular to the x-axis through point $1-c^2$. An example of the construction of this perpendicular is described in \cite[pages 79--80]{Martin}.
\end{proof}


Finally, using the sets of constructible points of maps in Table \ref{cataleg_conjuntsfinals}, we can deduce the following relations:

\begin{theorem} Arithmetical classification of tools:
\label{classificacioaritmeticaeines}
\begin{itemize}
\item[i)] $ \mathcal C\eqar\mathcal RC\eqar\mathcal EC\eqar \mathcal REC$.
\item[ii)] $ \mathcal O\eqar \mathcal CO$.
\end{itemize}
\end{theorem}

\section{Conclusions}

We have introduced a new formal language and illustrated it with the most common geometric instruments, even though a more extended study can be found in \cite{TramunsThesis}.  Some of the instruments presented there, such as the marked ruler or the marked ruler and compass, require a precise description of the {\em neusis} process and lead to interesting axioms, involving curves as the conchoid of Nicomedes or the Lima\c{c}on of Pascal. 

The language proposed has the advantage that it is open, in the sense that other instruments different that those we have considered can be studied and formalized in this way: it suffices to analyze the axioms they can perform and define them using curves and points. After that, the geometric and arithmetic relations between this instrument and other existing instruments can be studied.

Finally, another significant advantage of this language is the possibility of introducing virtual tools, that is, tools not necessarily attached to any physical instrument: we can choose some existing axioms and combine them to create a virtual tool. These kind of tools can be both interesting on their own and useful as auxiliary resources for studying known instruments.

\newpage
\section*{Annex}
\label{TableAnnex}
\begin{center}
\begin{tabular}{|p{6.5cm}|p{6.5cm}|}
\hline
Axiom & Description\\
\hline&\\
$\ell= $\textbf{Line}$(A,B)$ & Line through points $A$, $B$.
\\&\\
$c= $\textbf{Circle}$(A,B)$ & Circle with center $A$ through
 $B$.
\\&\\
$c= $\textbf{RadiusCircle}$(A,B,C)$ & Circle with center
$A$ and radius the distance~$BC$.    
\\&\\
$\ell_1,\ell_2= $\textbf{Bisector}$(\ell,\ell')$ & Bisectors of the angle formed by the
lines $\ell$ and $\ell'$.
\\&\\
$\ell= $\textbf{PerpendicularBisector}$(A,B)$ & Perpendicular bisector of the
segment $AB$.\\&\\
$\ell'= $\textbf{Perpendicular}$(\ell,P)$ & Line perpendicular to line $\ell$ passing through point $P$.
\\&\\
$\ell= $\textbf{PointPerpendicular}$(A,B,C)$ & Line perpendicular to the segment $AB$ through~$C$.
\\&\\
$\ell_1,\ell_2= $\textbf{Tangent}$(\ell,F,A)$ & Tangents through $A$ to the
parabola with directirx $\ell$ and focus $F$.
\\&\\
$\ell_1,\ell_2,\ell_3= $\textbf{CommonTangent}$(\ell,F,\ell',F')$ &
Common tangents to the parabola with directrix $\ell$ and
focus $F$ and the parabola with directrix $\ell'$ and focus $F'$.
\\&\\
$\ell= $\textbf{PerpendicularTangent}$(\ell_1,F,\ell_2)$ & The tangent line
 to the parabola with directrix $\ell_1$ and focus $F$ which is
perpendicular to the line $\ell_2$.
\\&\\
$c= $\textbf{Conic}$(\ell,F,A,B)$ & Conic with directrix $\ell$,
focus $F$ and excentricity the distance between $A$ and $B$.
\\&\\
\hline\hline&\\
$P= $\textbf{LineIntersect}$(\ell,\ell')$& Intersection point of lines
 $\ell$ and $\ell'$.
\\&\\
$P_1, P_2= $\textbf{CircleIntersect}$(c,c')$& Intersection points  of circles
 $c$ and $c'$.
\\&\\
$P_1, P_2= $\textbf{LineCircleIntersect}$(\ell,c)$ & Intersection points of line
 $\ell$ with circle $c$.
\\&\\
$P_1,...,P_4= $\textbf{ConicIntersect}$(c,c')$ & Intersection points of conics
 $c$, $c'$.
\\&\\
$P_1, P_2= $\textbf{LineConicIntersect}$(\ell,c)$ & Intersection points of line
$\ell$ with conic $c$.
\\&\\
$P_1,\dots,P_4= $\textbf{CircleConicIntersect}$(c,c')$ & Intersection points of circle
 $c$ with conic $c'$.
\\&\\
\hline

\end{tabular}

\end{center}

For the axioms generating several objects, one has to specify an ordering to properly identify each of them. When there is no natural ordering, we use a {\em radial sweep}, a common technique in computational geometry. It consists in sweeping counterclockwise the plane with a given half-line; the points are ordered in the order they are met. For the axioms described in this table we propose the following ordering:

\vskip 3mm
$\textbf{Bisector}$: Line $\ell_1$ bisects the oriented angle $\widehat{\ell\ell'}$ and $\ell_2$ bisects $\widehat{\ell'\ell}$.

\vskip 3mm
$\textbf{Tangent}$: The  lines $\ell_1$ and $\ell_2$ are ordered following a radial sweep with center $A$ and half-line $AF$.

\vskip 3mm
$\textbf{CommonTangent}$: The lines  $\ell_1$, $\ell_2$, $\ell_3$  are ordered according to the order of their contact points with the first parabola in a  radial sweep with center $F$ and half-line $FF'$.

\vskip 3mm
$\textbf{CircleIntersect}$: Let $C$ and $C'$ be the centers of circles $c$ and $c'$ respectively. The order of points $P_1$ and $P_2$ is given by a radial sweep with center $C$ and half-line $CC'$. We use the same criterion to order the output of axioms $\textbf{ConicIntersect}$ and $\textbf{CircleConicIntersect}$, taking as the center of a conic the midpoint of the segment defined by the focus.

\vskip 3mm
$\textbf{LineCircleIntersect}$: If $\ell$ is not a diameter of circle $c$, we order $P_1$ and $P_2$ to assure that the angle $P_2CP_1$ is positive. If the line is a diameter, we order points $P_1, P_2$ as points on the line $\ell$, using a radial sweep with center $C$ and half-line $CO$. We use the same criterion to order the output of the axiom $\textbf{LineConicIntersect}$.

\bibliographystyle{amsalpha}

\end{document}